\documentclass[12pt,letterpaper]{amsart}

\usepackage{palatino, euler, epic, eepic, xy, microtype, stmaryrd, epsfig, epstopdf,color, listings}
\usepackage{graphicx,float}%,subfig}

\xyoption{all}

\usepackage{enumerate}
\usepackage[a4paper,bindingoffset=0.2in,left=0.8in,right=1.2in,top=1in,bottom=1in,footskip=.25in,marginparwidth=0.8in]{geometry}
\usepackage{float}
\usepackage{caption}
\usepackage{amsmath, amssymb, amsfonts, amsthm, accents, xcolor, enumerate, esint}
\usepackage{mathtools}
\usepackage{tikz-cd}
\usepackage{hyperref}
\usepackage{empheq}
\usepackage{adjustbox}

\usepackage{amsmath, amssymb, amscd, verbatim, xspace,amsthm}
\usepackage{latexsym, epsfig, color}

\newcommand\coker{\operatorname{coker}}
\newcommand\image{\operatorname{image}}

\newcommand\tensor{\otimes}
\newcommand\isom{\cong}

\newcommand\bq{\begin{equation}}
  \newcommand\eq{\end{equation}}

\newtheorem{proposition}{Proposition}[section]
\newtheorem{theorem}[proposition]{Theorem}

\newtheorem{lemma}[proposition]{Lemma}

\theoremstyle{definition}
\newtheorem{definition}[proposition]{Definition}
\newtheorem*{nn-definition}{Definition}

\theoremstyle{remark}
\newtheorem{remark}[proposition]{Remark}
\usepackage{url}

\numberwithin{equation}{section}

\newcommand{\cut}[1]{}
\newcommand\hidden[1]{}

\newcommand{\FF}{\mathbb{F}}

\newcommand{\PP}{\mathbb{P}}
\newcommand{\QQ}{\mathbb{Q}}
\newcommand{\RR}{\mathbb{R}}

\newcommand{\setmin}{{\smallsetminus}}                                %
\newcommand{\ZZ}{{\mathbb{Z}}}                                        %
\newcommand{\cO}{{\mathcal O}}                                        %
\newcommand{\Bl}{\operatorname{Bl}}                                   %                                                         %
\newcommand{\Cl}{\operatorname{Cl}}                                   %                                                         %
\newcommand\ord{\operatorname{ord}}
\usepackage[makeroom]{cancel}                                                               %

\newcommand\torsion{\operatorname{Torsion}}
\newcommand\tors{\operatorname{Tors.}}
\usepackage{caption}                                                  %
\usepackage{subcaption}                                               %
\usepackage{tikz-cd}                                                  %
\usetikzlibrary{calc}                                                 %
\usetikzlibrary{positioning}                                          %
\usetikzlibrary{decorations.pathreplacing}                            %
\usetikzlibrary{angles}                                               %
\usetikzlibrary{quotes}                                               %
\usetikzlibrary{shapes.misc}                                          %
\usetikzlibrary{decorations.text}                                     %
\usetikzlibrary{arrows}                                               %
\usetikzlibrary{spy}
\usetikzlibrary{intersections}
\usetikzlibrary{through}
\usetikzlibrary{backgrounds}

\usetikzlibrary{arrows.meta}

\usepackage{pgfplots}
\definecolor{pal1}{RGB}{215,48,39}
\definecolor{pal2}{RGB}{252,141,89}
\definecolor{pal3}{RGB}{254,224,144}
\definecolor{pal4}{RGB}{145,191,219}
\definecolor{pal5}{RGB}{69,117,180}
\pgfkeys{/pgf/number format/relative round mode=fixed}
\pgfplotsset{
  contour/label node code/.code={
    \node{\pgfmathprintnumber{#1}\,ms};
  }
}

\title{Non-existence of negative curves}

\author{Javier Gonz\'alez Anaya, Jos\'e Luis Gonz\'alez and Kalle Karu}
\address{
  J. Gonz\'alez-Anaya, Department of Mathematics, University of California, Riverside,
  Riverside, CA 92521, United States.  \newline \indent
  J.L. Gonz\'alez, Department of Mathematics, University of California, Riverside,
  Riverside, CA 92521, United States.  \newline \indent
  K. Karu,
  Department of Mathematics, University of British Columbia, 
  Vancouver, BC V6T1Z2, Canada.} 
\email{javiergo@ucr.edu, jose.gonzalez@ucr.edu, karu@math.ubc.ca}
\thanks{The second author was supported by a grant from the Simons Foundation (Award Number 710443) and by the UCR Academic Senate. The third author was supported by a NSERC Discovery grant.}

\begin{document}
\begin{abstract}
Let $X$ be a projective toric surface of Picard number one blown up at a general point. We bring an infinite family of examples of such $X$ whose Kleiman-Mori cone of curves is not closed: there is no negative curve generating one of the two boundary rays of the cone. These examples are related to Nagata's conjecture and rationality of Seshadri constants.
\end{abstract}
\maketitle
\setcounter{tocdepth}{1} % this just includes subsections

% \tableofcontents

% ************************************************************************************************************************

%%%%%%%%%%%%%%%%%%%%%%%%%%%%%%%%%%%%%%%%%%%%%%%%%%%%%%%%%%%%%%%%%%%%%%%%%%%%%%%%%%%% 
%%%%%%%%%%%%%%%%%%%%%%%%%%%%%%%%%%%%%%%%%%%%%%%%%%%%%%%%%%%%%%%%%%%%%%%%%%%%%%%%%%%% 
%%%%                                                                            %%%% 
%%%% Section 1. Introduction                         %%%%
%%%%                                                                            %%%% 
%%%%%%%%%%%%%%%%%%%%%%%%%%%%%%%%%%%%%%%%%%%%%%%%%%%%%%%%%%%%%%%%%%%%%%%%%%%%%%%%%%%% 
%%%%%%%%%%%%%%%%%%%%%%%%%%%%%%%%%%%%%%%%%%%%%%%%%%%%%%%%%%%%%%%%%%%%%%%%%%%%%%%%%%%% 

\section{Introduction}

Let $X_\Delta$ be a projective toric surface defined by a rational triangle $\Delta\subseteq \RR^2$, and let $X$ be the blowup of $X_\Delta$ at a general point $e$. In this article we consider the question of whether the Kleiman-Mori cone of curves of $X$ is closed. The cone of curves of $X$ is two-dimensional, with one extremal ray generated by the class of the exceptional curve $E$ and the other extremal ray generated by a class $\gamma$ of nonpositive self-intersection. If we can choose $\gamma$ to be the class of an irreducible curve $C$, we call this $C$ a negative curve of $X$. In other words, a negative curve is any irreducible curve $C\subseteq X$ of nonpositive self-intersection, $C^2\leq 0$, different from $E$. Our main result is to bring a family of examples of $X$ that have no negative curves, and hence their cones of curves are half-open.

\begin{theorem} \label{thm-main}
Let $\Delta$ be the triangle with vertices 
\[ (\alpha,0), (\alpha+9/7,0), (3,7), \]
where $\alpha$ is a rational number in the interval $[-4/15, -16/105]$, and let $X=\Bl_e X_\Delta$ be defined over any algebraically closed field of characteristic $0$. Then, $X$ contains no negative curve.
\end{theorem}

The set of varieties $X_\Delta$ defined by a rational triangle contains as a subset all weighted projective planes $\PP(a,b,c)$. The examples in the theorem are not in this subset. However, any $X_\Delta$ defined by a rational triangle is a quotient of $\PP^2$ by a finite abelian group. 

Let us put this result in the context of \cite{Geography}. In that article we considered all rational triangles $\Delta$ and catalogued all known negative curves. Part of this work is summarized in Figure~\ref{fig-plot}. We consider triangles $\Delta$ with a horizontal base and the left and right edges with slopes $s$ and $t$. Then, a triangle $\Delta$, or the blowup $X=\Bl_e X_\Delta$, corresponds to a point in the $st$-plane. 
We may and will assume that $e=(1,1)$ is the identity point in the torus $T\subset X_\Delta$. 
A negative curve is defined by a Laurent polynomial $\xi(x,y)$ that vanishes to order $m$ at the point $e$ and whose Newton polygon fits into a triangle $\Delta$ of area $\leq m^2/2$. For each polynomial $\xi$ there may be many such triangles $\Delta$, hence the same polynomial $\xi$ can define a negative curve in many varieties $X$. The set of all such $X$ for a fixed $\xi$ is a region in the $st$-plane that typically has the shape of a diamond. 
In Figure~\ref{fig-plot}, different diamonds correspond to different polynomials $\xi$. If some $X$ contains a negative curve with strictly negative self-intersection, then this negative curve is unique in $X$. This explains why the diamonds can intersect only at their boundary points, in which case $X$ contains two negative curves of zero self-intersection. Parts of the plane are not covered by any diamond. The varieties $X$ corresponding to these points are expected to contain no negative curve. The varieties $X$ in Theorem~\ref{thm-main} correspond to points in the intersection of the diamonds labeled with $\xi$ and $\zeta$.

\begin{figure}[t]
  \centering
  \includegraphics[width=\textwidth]{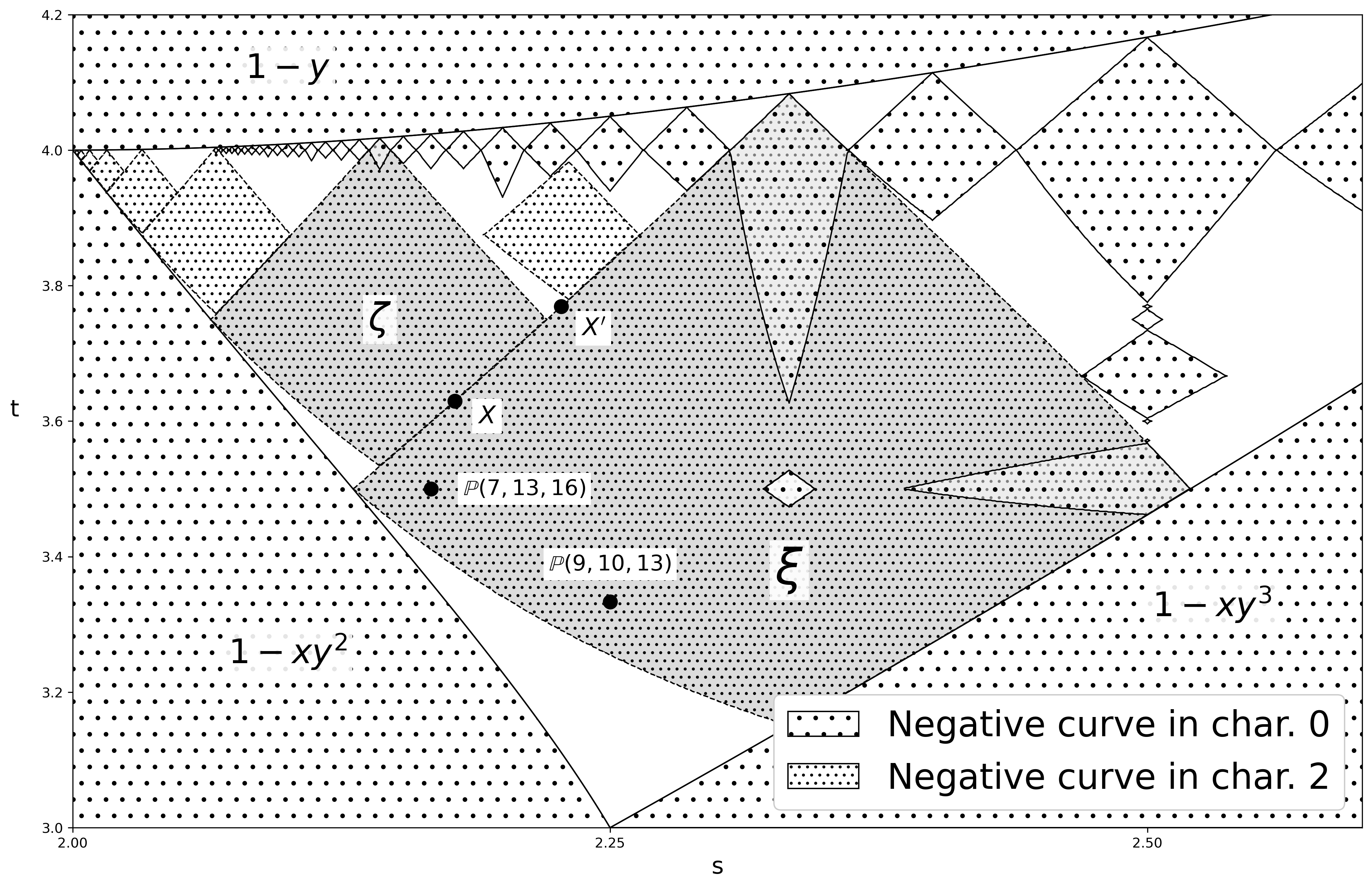}
  \caption{Negative curves in the parameter space of triangles.}
  \label{fig-plot}
\end{figure}

We will use results in characteristic $2$ to prove Theorem~\ref{thm-main}.
Toric varieties $X_\Delta$ can be defined over a field of any characteristic. The picture of diamonds for each field only depends on the characteristic and not on the field. A polynomial $\xi$ that defines a negative curve in characteristic $0$ can be chosen to have integer coefficients. Its reduction mod $p$ then defines a (possibly reducible) negative curve in characteristic $p$. Going from characteristic $0$ to characteristic $p$, the diamonds may get larger, and there will be new diamonds. Figure~\ref{fig-plot} shows examples of diamonds of negative curves in characteristic $0$ and a few additional diamonds of negative curves in characteristic $2$. The points corresponding to $X$ in Theorem~\ref{thm-main} lie in the intersection of two such diamonds that only exist in characteristic $2$. This means that $X$ in characteristic $2$ contains two disjoint irreducible curves $C$ and $D$, both with self-intersection zero. These curves span an extremal ray of the cone of curves of $X$. In fact, if $\xi(x,y)$ and $\zeta(x,y)$ are the defining equations of the two curves, then polynomials in $\FF_2[\xi,\zeta,x^{\pm 1}]$ give us all curves on this ray. A negative curve for $X$ in characteristic $0$ must also have self-intersection $0$ and hence lie on the same extremal ray (otherwise its reduction mod $2$ would give a strictly negative curve in characteristic $2$). We prove that no polynomial in $\FF_2[\xi,\zeta, x^{\pm 1}]$ (except the monomials $x^l$) can be lifted to characteristic $0$ while preserving the condition of vanishing at $e$ to order $m$.

Let us mention three open problems that are related to the existence of negative curves: the Mori Dream Space property in positive characteristic, irrationality of Seshadri constants, and Nagata's conjecture. The triangles in   Theorem~\ref{thm-main} do not give us new information about these problems. However, a similar study of more general triangles would solve the open problems.

It is known by a result of Cutkosky \cite{Cutkosky} (see also the algebraic version by Huneke \cite{Huneke}) that $X$ is a Mori Dream Space (MDS) if and only if it contains a negative curve $C$ and a curve $D$ disjoint from it. This result holds even when $C$ and $D$ have self-intersection zero, see \cite[Lemma~2.1]{GGK1}. Thus, the varieties $X$ in Theorem~\ref{thm-main} are MDS in characteristic $2$. In positive characteristic, it is also known \cite{Cutkosky} that if $X$ contains a negative curve $C$ with strictly negative self-intersection, then $X$ is a MDS. This fact uses the result of Artin that such negative curves are contractible \cite{Artin62}. However, if $C$ has self-intersection zero, then it is not known if $X$ in positive characteristic is a MDS or not. In fact, we do not know any blowup $X$ as above that is not a MDS in positive characteristic. Let $\Delta'$ be the triangle as in Theorem~\ref{thm-main} but with $\alpha=-1/7$. Then $X' = \Bl_e X_{\Delta'}$ in characteristic $2$ contains a negative curve $C$ of self-intersection zero, defined by the same polynomial $\xi$ as above. We conjecture that $X'$ contains no curve $D$ disjoint from $C$ and hence is not a MDS. In Figure~\ref{fig-plot}, the variety $X'$ lies on the boundary of only one diamond.

Negative curves in $X$ are related to Seshadri constants (see the discussion in \cite{Castravet1,CutkoskyKurano}). Recall that if $L$ is an ample divisor on a variety $Y$ and $e\in Y$ is a point, then the single point Seshadri constant of $Y, L, e$ is 
\[ \epsilon(Y, L, e) = \inf_{e\in C} \frac{L\cdot C}{\operatorname{mult}_e C}.\]
Here the infimum is over all curves $C$ in $Y$ passing through $e$.
It is an open problem to find examples of surfaces where the single point Seshadri constant is irrational (see for example \cite{primer, DumnickiRat, HanumanthuHarbourne} for a survey and more recent work). In our case, $Y=X_\Delta$, $L$ any ample divisor on $Y$, the single point Seshadri constant at $e$ is the slope (with respect to $E$ and $L$) of the extremal ray in the cone of curves of $X=\Bl_e Y$.  If $X$ contains a negative curve then the slope is clearly rational. The examples $X$ in Theorem~\ref{thm-main} also have rational slopes because such an $X$ contains a negative curve in characteristic $2$ and the cones in characteristic $0$ and $2$ are equal. In general, if $X$ does not contain a negative curve, then the extremal ray is generated by a class $\gamma$ of self-intersection zero. The slope of this ray can be computed from the triangle, and it is typically irrational. When $X_\Delta = \PP(a,b,c)$, then the slope is rational if and only if $abc$ is a square in $\ZZ$. Examples such as the blowups of $\PP(9,10,13)$, $\PP(7,13,16)$ are conjectured to contain no negative curve \cite{KuranoMatsuoka}. Such weighted projective planes would then provide examples with irrational Seshadri constants.

Nagata's conjecture \cite{Nagata} states that if $p_1,\ldots,p_r$ are very general points in $\PP^2$ and $C\subseteq \PP^2$ is a curve that has multiplicity $m_i$ at the point $p_i$, then for $r > 9$,
\begin{equation} \label{eq-Nagata}
 \deg C > \frac{1}{\sqrt{r}} \sum_i m_i.
 \end{equation}
Nagata \cite{Nagata} proved the conjecture in the case where $r$ is a square. A slightly weaker inequality was proved by Xu \cite{xu}. 
A toric variety $X_\Delta$ given by a triangle can always be presented as a quotient of $\PP^2$ by a finite abelian group. The examples $X_\Delta$ in Theorem~\ref{thm-main} are quotients by an abelian group of order $r$ that is a square. Then a curve $\overline{C}$ in $X_\Delta$ that has multiplicity $m$ at $e$ can be pulled back to $\PP^2$ so that it passes through the $r$ preimages of $e$ with multiplicity $m$. Equation~\ref{eq-Nagata} applied in this situation says that the strict transform $C$ of $\overline{C}$ in $X=\Bl_e X_\Delta$ has positive self-intersection. In other words, if the inequality holds then $X$ has no negative curve. Of course, the $r$ points in this case are not very general in the sense of Nagata, so Nagata's conjecture, which is known in the case where $r$ is a square, does not imply Theorem~\ref{thm-main}. However, it is true that Theorem~\ref{thm-main} is equivalent to Nagata's inequality for the special set of points in the preimage of $e$. For the converse direction, assume that a curve $C\subseteq \PP^2$ fails the inequality in (\ref{eq-Nagata}). Then the union of the translates of $C$ by the group action is the inverse image of a curve in $X_\Delta$ whose strict transform in $X$ has nonpositive self-intersection. 

The proof of Theorem~\ref{thm-main} is organized as follows. We fix one triangle $\Delta$ and variety $X$ as in the statement of Theorem~\ref{thm-main}. In Section~\ref{sec-2} we prove the claim that polynomials in the ring $R=\FF_2[\xi,\zeta,x^{\pm 1}]$ give us all curves of self-intersection zero in the variety $X$ in characteristic $2$. In Section~\ref{sec-4} we show that none of the curves in characteristic $2$ can be lifted to characteristic $0$. We show that the obstruction to this lifting lies in a space $M$ of first cohomologies. Section~\ref{sec-3} defines $M$, gives it the structure of an $R$-module, and proves that it is a free module. The freeness is then used in Section~\ref{sec-4} to show that the relevant obstruction to lifting does not vanish. In Section~\ref{sec-5} we present some linear algebra computations. This is the only section where the proofs depend on the specific triangle $\Delta$; all other proofs work for any triangle in the intersection of the diamonds of $\xi$ and $\zeta$.

We work over the field $\FF_2$ throughout the paper. The main result, Theorem~\ref{thm-main}, holds over a field of characteristic zero. The background material about toric varieties such as their Weil divisors, class group, vanishing theorems, applies over any ground field.

\section{The section ring $R$} \label{sec-2}

Let $X$ be a normal variety and $D_1,\ldots, D_n$ Weil divisors on $X$. Define the section ring
\[ R= \bigoplus_{m_1,\ldots, m_n \in \ZZ} H^0\Big(X, \cO_X\Big(\sum_i m_i D_i\Big)\Big).\]
If the classes of $D_1, \ldots, D_n$ span the class group $Cl(X)_\QQ$, then $R$ is called a Cox ring of $X$ \cite{HuKeel}. In this section we consider a particular section ring of $X=\Bl_e X_\Delta$ and compute this ring explicitly.

\subsection{Weil divisors on $X_\Delta$.}

Toric varieties and their $T$-invariant divisors can be defined over any field (and even over any ring). We start with an arbitrary field, but then specialize to the case of $\FF_2$. Let us first recall how to find the class group of a toric variety using its $T$-invariant Weil divisors \cite{Fulton}. 

Let $\Delta$ be a rational triangle in $\RR^2$, for example a triangle as in Theorem~\ref{thm-main}. The toric variety $X_\Delta$ is $\QQ$-factorial, meaning that every Weil divisor has a nonzero integer multiple that is Cartier. We can compute the intersection number $C\cdot D$ of two curves $C,D \subseteq X_\Delta$ by reducing it to the intersection number of a curve and a Cartier divisor.

The class group of $X_\Delta$ is a direct sum of a free group of rank $1$ and a finite group. To construct it, one starts with $T$-invariant Weil divisors ($T$-Weil divisors for short), that means, integral linear combinations of the three boundary divisors. The group of $T$-Weil divisors is generated by nef divisors which have a geometric description. A nef $T$-Weil divisor corresponds to a triangle with sides parallel to the sides of $\Delta$ and all sides integral in the sense that the affine line spanned by each side contains a lattice point. Addition is defined as the Minkowski sum of triangles. 

Two triangles correspond to linearly equivalent divisors if they differ by an integral translation. The class group $Cl(X_\Delta)$ of $X_\Delta$ is the quotient of the group of $T$-Weil divisors by the linear equivalence relation. 

A $T$-Weil divisor defines the sheaf $\cO_{X_\Delta}(D)$ on $X_\Delta$. This sheaf is locally free if and only if $D$ is Cartier, that means, if the corresponding triangle is integral. More precisely, the sheaf is locally free everywhere except at the $T$-fixed points corresponding to non-integral vertices of the triangle.  Below we will see several cases where a morphism is an isomorphism or a sequence is exact if the divisors involved are Cartier. In the Weil case this may fail because of a torsion sheaf supported at the $T$-fixed points. (By a torsion sheaf we always mean a coherent torsion sheaf.)

Because a $T$-Weil divisor is supported in the complement of the torus $T$, we have a canonical trivialization of the sheaf $\cO_{X_\Delta}(D)$ on the torus. We can therefore view sections of the sheaf over an open $T$-invariant subset as consisting of Laurent polynomials in $x, y$ with certain vanishing orders along the three boundary divisors.

Given $T$-Weil divisors $D_1$ and $D_2$ in $X_\Delta$, there is an inclusion of sheaves
\[ \cO_{X_\Delta}(D_1) \otimes_{\cO_{X_\Delta}} \cO_{X_\Delta}(D_2) \subseteq \cO_{X_\Delta}(D_1 + D_2).\]
The inclusion may not be an equality, but the cokernel is a torsion sheaf supported at the $T$-fixed points. The combinatorial reason why the inclusion may not be an equality is that lattice points in the Minkowski sum of two polyhedra may not be sums of lattice points from the two polyhedra. In the case above, working locally in a $T$-invariant open subset, the polyhedra are rational translates of a cone in $\RR^2$.

Given a global section 
\[ \xi\in H^0(X_\Delta, \cO_{X_\Delta}(D_1) ),\]
we let $\overline{C}=V(\xi)$ be the closure of the curve that $\xi$ defines away from the three $T$-fixed points. For any $T$-Weil divisor $D_2$, multiplication with $\xi$ defines a morphism
\[ \xi: \cO_{X_\Delta}(D_2) \to \cO_{X_\Delta}(D_1 + D_2).\]

\begin{lemma} \label{lem-tor}
Assume that $\overline{C}$ is irreducible and intersects the torus $T$. Then for any $D_2$ the morphism $\xi$ is injective with cokernel
\[ \coker(\xi) = \cO_{X_\Delta}(D_1 + D_2)|_{\overline{C}}/\torsion.\]
\end{lemma}

\begin{proof}
The ideal sheaf of $\overline{C}$ is 
\[ I_{\overline{C}}= \xi \cdot \cO_{X_\Delta}(-D_1) \subseteq \cO_{X_\Delta}. \]
Tensoring this with $\cO_{X_\Delta}(D_1 + D_2)$ gives
\[ I_{\overline{C}} \tensor \cO_{X_\Delta}(D_1 + D_2) = \xi \cdot \cO_{X_\Delta}(-D_1) \tensor \cO_{X_\Delta}(D_1 + D_2) \subseteq \xi \cdot\cO_{X_\Delta}(D_2).\]
The quotients of $\cO_{X_\Delta}(D_1 + D_2)$ by these subsheaves now give a surjective morphism
\small
\[ \cO_{X_\Delta}(D_1 + D_2)|_{\overline{C}} = \cO_{X_\Delta}(D_1 + D_2)/ I_{\overline{C}} \tensor \cO_{X_\Delta}(D_1 + D_2) \to \cO_{X_\Delta}(D_1 + D_2)/ \xi \cdot \cO_{X_\Delta}(D_2) = \coker(\xi).\]
\normalsize
This map is an isomorphism away from the $T$-fixed points, so its kernel is a torsion sheaf. Let us check that $\coker(\xi)$ has no torsion; hence the kernel consists of all the torsion. 

We consider a $T$-invariant open affine $U\subseteq X_\Delta$. Let $s$ be a section of $\cO_{X_\Delta}(D_1 + D_2)$ on $U$ such that for some section $f$ of $\cO_{X_\Delta}$ that does not vanish on $\overline{C}$, $f\cdot s$ lies in the subsheaf $\xi \cdot \cO_{X_\Delta}(D_2)$. Recall that we have trivialized all sheaves over the torus, hence we may view all sections as Laurent polynomials. Now, the polynomial $\xi$ divides $f\cdot s$. Since $\xi$ is irreducible and not a unit, it divides $f$ or $s$. In the first case $f$ vanishes on $\overline{C}$, which is not allowed. It follows that $\xi$ must divide $s$. The quotient polynomial is a section of $\cO_{X_\Delta}(D_2)$. This last claim follows because sections of $\cO_{X_\Delta}(D_2)$ are Laurent polynomials with certain vanishing orders at the three boundary divisors. One checks that the quotient has the required orders of vanishing.
\end{proof}

\subsection{Weil divisors on $X$.}

We now fix one triangle $\Delta$ as in Theorem~\ref{thm-main} and compute the class groups of $X_\Delta$ and $X=\Bl_e X_\Delta$ explicitly. These class groups do not depend on the field over which the varieties are defined.

The triangles defining nef $T$-Weil divisors on $X_\Delta$ are determined by their lower left and right vertices $(x_L, y)$, $(x_R,y)$. The integrality condition of the Weil divisors implies that $x_L\in \frac{1}{n_L}\ZZ$, $x_R \in \frac{1}{n_R} \ZZ$ and $y\in \ZZ$ for some nonzero integers $n_L, n_R$. The numbers $n_L, n_R$ are the numerators of the slopes $s=n_L/d_L$, $t=n_R/d_R$ written in lowest terms.
The set of triples 
\[ ( x_L, x_R, y) \in \frac{1}{n_L} \ZZ \times \frac{1}{n_R} \ZZ \times \ZZ\]
is the group of $T$-Weil divisors of $X_\Delta$. 

\begin{lemma} For any triangle $\Delta$ in Theorem~\ref{thm-main}, either $n_L\mid n_R$ or $n_R \mid n_L$.
\end{lemma}

\begin{proof}
Let $\alpha=p/q$ in lowest terms. Then the slopes $s$ and $t$ are
\[ s= \frac{7q}{3q-p}, \quad t= \frac{7^2q}{7(3q-p)-9q},\]
not necessarily in lowest terms.

If $7 \nmid q$ then $n_R=7^2 q$, hence $n_L \mid 7q \mid n_R$.
If $7 \mid q$ then $7 \nmid 3q-p$ and hence $n_L = 7q$. We can divide the numerator and denominator of $t$ by $7$. Then $n_R \mid 7q = n_L$.
\end{proof}

We will assume below that $n_L \mid n_R$. If $n_L \nmid n_R$, we can change the roles of the lower left and right vertices. 
With this assumption, the group of $T$-Weil divisors of $X_\Delta$ can be expressed as the set of triples

\[ (b,  x_L, y) \in \frac{1}{n_R} \ZZ \times  \frac{1}{n_L} \ZZ \times \ZZ,\]
where $b = x_R-x_L$ is the length of the base.
The class group of $X_\Delta$ is the quotient by the group $\ZZ\times\ZZ$ acting on $(x_L, y)$ by integral translation.  Thus
\[ \Cl (X_\Delta) \isom \left\{(b, x_L) \in  \frac{1}{n_R} \ZZ \times  \left(\frac{1}{n_L} \ZZ \right)/\ZZ\right\}.\]
The numerical equivalence class of a nef divisor is given by the size of the triangle, that means by $b$.

Let $X=\Bl_e X_\Delta$. Then its class group is 
\[ \Cl(X) = \Cl(X_\Delta) \times \ZZ (-E),\]
where $E$ is the class of the exceptional divisor. We write a class as a triple $(b,x_L, m)$.

For the rest of this section we work over the field $\FF_2$. We start by constructing two irreducible disjoint curves $C$ and $D$ in $X$ with zero self-intersection. Recall that we have fixed a variety $X=\Bl_e X_\Delta$ as in Theorem~\ref{thm-main}, where the triangle $\Delta$ is determined by its left vertex $(\alpha,0)$. Let $\pi:X\to X_\Delta$ be the blowup morphism.

\begin{lemma} \label{lem-xi}
Let $H$ be the $T$-Weil divisor on $X_\Delta$ corresponding to the triangle $\Delta$. Then there exists a nonzero section 
\[\xi\in H^0(X, \cO_X(\pi^* H-3E)) \]
that defines an irreducible curve $C \subseteq X$ with zero self-intersection. The curve $C$ passes through the $T$-fixed points corresponding to the lower left and right vertices of $\Delta$.
\end{lemma}

\begin{proof}
Let $\xi$ be the reduction mod $2$ of the polynomial
\[ \tilde{\xi}(x,y) = 1+ x(1+y + y^2) + x^2 y^4 + x^3 y^7 \in \ZZ[x,y].\]
The Newton polygon of $\xi$ is the triangle with vertices $(0,0), (1,0), (3,7)$.
This triangle lies in $\Delta$ as required. It also follows that $\xi$ is irreducible because its Newton polygon has sides of lattice length $1$, and hence is not the Minkowski sum of smaller integral polygons \cite[Lemma 2.1]{GGK2}.  
When changing coordinates, we get
\[ \tilde{\xi}(x+1, y+1) = (4 x^2 + 32 x y + 28 y^2 + 8 x + 14 y + 6) + (x ^3 + 25 x^2 y + 76 x y^2 + 39 y^3) + \text{h.d.t.}\]
Clearly all terms of degree less than $3$ vanish mod $2$, hence $\xi$ defines a global section of $\cO_X(\pi^* H-3E)$. Since $\xi$ has multiplicity exactly $3$ at $e$, the curve $C\subseteq X$ that it defines does not have $E$ as a component and hence is irreducible. The self-intersection number of $C$ is zero:
\[ C\cdot C = 2 Area(\Delta) - m^2 = 0.\]
Because the triangle $\Delta$ is not integral at the  lower two vertices, the curve $C$ passes through the corresponding $T$-fixed points.
\end{proof}

\begin{lemma} \label{lem-zeta}
Let $H'$ be the $T$-Weil divisor on $X_\Delta$ corresponding to the triangle with lower vertices $(0,0), (3,0)$. Then there exists a nonzero section 
\[\zeta\in H^0(X, \cO_X(\pi^* H'-7E)) \]
that defines an irreducible curve $D \subseteq X$ with zero self-intersection. The curve $D$ passes through the $T$-fixed point corresponding to the top vertex of $\Delta$. 
\end{lemma}

\begin{proof}
We will construct the polynomial $\zeta$ with Newton polygon equal to the triangle with vertices $(0,0), (3,0), (7,15)$. This triangle lies in the triangle with the same base and sides parallel to $\Delta$. Since the left and right sides of the Newton polygon have lattice length $1$, it follows that such $\zeta$ is irreducible, and since the top vertex of the triangle with sides parallel to $\Delta$ is not integral (it has $y$-coordinate $49/3$), the curve $D$ defined by $\zeta$ passes through the corresponding $T$-fixed point. The self-intersection number of $D$ is zero as before.

Let $\xi$ be as in the previous lemma and define $\zeta$ to be
\[ \zeta(x,y) = x(y-1)\xi^2 + (xy^2-1)^7.\]
Both terms in the sum have all nonzero monomials supported in the triangle defined above, except the monomial $x^7y^{14}$. However, these terms cancel in the sum. This shows that $\zeta$ is supported in the correct triangle. Clearly both terms of the sum vanish to order $7$ at $e=(1,1)$, hence $\zeta$ defines a global section of $\cO_X(\pi^* H'-7E)$. One checks that the coefficients of the monomials corresponding to the vertices $(0,0), (3,0),(7,15)$ are nonzero, which implies that $\zeta$ is nonzero and has the stated Newton polygon.

We also need to check that the curve $D$ in $X$ does not contain $E$ as a component. In other words, that $\zeta$ vanishes to order exactly $7$ at $e$. If the vanishing order were greater, then $\zeta$ would define a negative curve in $X$ with strictly negative self-intersection, and that would contradict the existence of the irreducible curve of zero self-intersection defined by $\xi$.
\end{proof}

Figure~\ref{fig-plot} shows the diamonds of the polynomials $\xi$ and $\zeta$. The point corresponding to any $X$ as in Theorem~\ref{thm-main} lies in the intersection of these two diamonds. It follows from Theorem~\ref{thm-R} below that $\xi$ and $\zeta$ are the only nonzero global sections of their respective sheaves. This gives the uniqueness of the curves $C$ and $D$ in their linear equivalence classes. 

For later use, let us consider $\tilde{\xi} \in \ZZ[x,y]$, the lifting of $\xi$, and write    
\[ \tilde{\xi}(x+1, y+1) = 2f_1 + g_1,\]
where $f_1$ consists of terms of degree less than $3$ and $g_1$ of degree $3$ or greater. The proof of Lemma~\ref{lem-xi} shows that
\begin{align*} f_1 & \equiv y+1 \pmod 2,\\
 g_1 & \equiv x^3 + x^2y +y^3 + \text{ higher degree terms} \pmod 2. 
 \end{align*} 
 Similarly, we lift $\zeta$ to a polynomial $\tilde{\zeta}\in \ZZ[x,y]$ and write
\[ \tilde{\zeta}(x+1,y+1) = 2 f_2 + g_2,\]
where $f_2$ has terms of degree less than $7$ and $g_2$ of degree at least $7$. The lowest degree terms of $g_2$ mod $2$ lie in degree $7$.

Let us study the curve $C = V(\xi)\subseteq X$ further. Its intersection number with the exceptional divisor $E$ is $3$. This intersection is defined by the vanishing of the lowest degree terms in $\xi(x+1,y+1)$, namely the homogeneous polynomial $x^3 + x^2y +y^3$. If we dehomogenize it, we get a single variable irreducible polynomial $t^3+t+1 \in \FF_2[t]$. Thus, if we consider $C$ as a scheme over $\FF_2$, its intersection with $E$ is a single point with residue field $\FF_8$. On the other hand, if we consider $C$ as a variety over $\overline{\FF}_2$, then this point splits into three distinct points. Since these are simple points on $E$, it follows that $C$ is smooth at these points and intersects $E$ transversely.

\subsection{A section ring of $X$.}

We continue working over $\FF_2$. Consider the extremal ray of the cone of curves of $X$ that contains $C$ and $D$. The elements of $\Cl(X)$ that have their numerical equivalence classes on this ray satisfy  
\[ b = \frac{3}{7} m \geq 0.\]
The equality here is the condition that $m^2$ equals twice the area of the triangle, hence the class has self-intersection zero. The inequality restricts this line to the ray of nef divisors.
We may write this subset of $Cl(X)$ as
\[ \left\{ (b, x_L, m) \in \Cl(X) \,|\, b = \frac{3}{7} m \geq 0\right\} \isom
 \left\{ (x_L,m) \in \left(\frac{1}{n_L} \ZZ\right)/\ZZ \times \ZZ_{\geq 0}\right\}.\]
Let us choose Weil divisors $D_1, D_2$ such that the classes of $\pm D_1$ and $D_2$  generate this semigroup:
\[ [D_1] = \left(x_L=\frac{1}{n_L}, m=0\right), \quad [D_2]=  (x_L=0, m=1).\]
The triangle of the divisor $D_1$ is the point $(1/n_L, 0)$. The triangle of $D_2$ has lower vertices $(0,0), (3/7,0)$.

\begin{definition} \label{def-R}
Define the section ring 
\begin{align*} R &= \bigoplus_{m_1\in \ZZ, m_2 \in \ZZ_{\geq 0}} H^0(X, \cO_X(m_1 D_1 + m_2 D_2)) \\
 &= \bigoplus_{x_L\in \frac{1}{n_L} \ZZ, m\in \ZZ_{\geq 0}} H^0(X, \cO_X(x_L D_1 + m D_2)).
 \end{align*}
\end{definition}

More concretely, $R$ is a graded ring. To find its elements of degree $(x_L, m)$, we consider the triangle with left vertex $(x_L, 0)$ and base of length $\frac{3}{7}m$. Then the homogeneous part of degree $(x_L, m)$ in $R$ consists of all Laurent polynomials in $\FF_2[x^{\pm 1},y^{\pm 1}]$ that are supported in the triangle and vanish to order at least $m$ at $e=(1,1)$. 

Note that $\xi$ and $\zeta$ are homogeneous elements of $R$. The element $\xi$ lies in degree $(x_L= \alpha, m=3)$, and $\zeta$ lies in degree $(x_L=0, m=7)$. The polynomial $x$ lies in degree $(x_L=1, m=0)$. 

\begin{theorem} \label{thm-R}
The elements $\xi$, $\zeta$ and $x$ of $R$ are algebraically independent over $\FF_2$ and
\[ R \isom \FF_2 [\xi, \zeta, x^{\pm 1}].\]
\end{theorem}

\begin{proof}
Let $s$ be a homogeneous element of $R$ in degree $(x_L, m)$. If the top vertex of the triangle corresponding to this degree does not lie in the Newton polygon of $s$, then the curve $V(s)$ passes through the corresponding $T$-fixed point. We know that $D= V(\zeta)$ also passes through this point. Since $D\cdot V(s) = 0$ and $D$ is irreducible, it follows that $\zeta$ divides $s$ as a Laurent polynomial, and hence it divides $s$ in the ring $R$. 

Similarly, if the Newton polygon of $s$ does not contain the left or the right vertex of the triangle, then $\xi$ divides $s$ in $R$. 

If all three vertices lie in the Newton polygon of $s$, then the triangle must be integral. This implies that $x_L=l\in \ZZ$ and $m=7n$ for some $n\in \ZZ$. Now, $x^l \zeta^n$ is supported in the same triangle and the difference $s-x^l \zeta^n$ has Newton polygon that does not contain the left and right vertices; hence we can continue dividing by $\xi$. 

By induction, this process stops when $m=0$. In this case the triangle is reduced to a point with coordinates $(x_L, y=0)$. In this degree there is a nonzero global section if and only if the point is integral, $x_L=l\in\ZZ$. Then $x^l$ is the unique nonzero global section. 

The previous argument shows that every element $s\in R$ can be written as a polynomial in $\xi,\zeta, x^{\pm 1}$. Hence the graded morphism 
\[ \FF_2 [\xi, \zeta, x^{\pm 1}] \to R\]
is surjective. We will show that it is also injective. Since the morphism is graded, its kernel must be homogeneous. Suppose the map in degree $(x_L, m)$ is not injective. The degree $(x_L, m)$ part of the left hand side has a basis consisting of all monomials
\[ x^l\xi^i\zeta^j, \quad l,i,j\in \ZZ, i,j \geq 0\]
such that $3i+7j=m$ and $l+\alpha i = x_L$. We will consider the Newton polygons of these monomials restricted to the $x$-axis, $y=0$. The Newton polygons of $\xi$ and $\zeta$ restrict to intervals $[0,1]$ and $[0,3]$, respectively. The Newton polygon of $x$ is simply the point $1$. From this we can compute the Newton polygon of the monomial to be
\[ [l,l]+[0,i+3j] = \left[x_L-\alpha i , x_L- \alpha i \right] + [0,i+3j].\]
Notice that the left endpoints of these intervals increase with $i$ since $\alpha <0$. Hence if there is a relation among the monomials, then one proves by induction on $i$ that the coefficients in front of the monomials must all be zero.
\end{proof}

\section{The module $M$ of first cohomologies.} \label{sec-3}
We fix a variety $X$ as in Theorem~\ref{thm-main}. 
In the previous section, we defined the ring $R$ of zeroth cohomologies of certain sheaves. In this section we define similarly the space $M$ of first cohomologies of the same sheaves. This $M$ has the structure of an $R$-module. In the next section we prove that no element of $R$ can be lifted to characteristic zero. The obstruction to lifting lies in the module $M$. 

We work over the field $\FF_2$ throughout this section. The computation of cohomology in the next subsection works over any field.

\subsection{The first cohomology} 
Let $H$ be a $T$-Weil divisor on $X_\Delta$, $m\geq 0$, and $\pi: X\to X_\Delta$ the blowup morphism. We consider the divisor 
\[ D = \pi^* H -mE \]
on $X$. Global sections of $\cO_X(D)$ can be identified with global sections of $\cO_{X_\Delta}(H)$ that vanish to order at least $m$ at $e=(1,1)$. Let us write this more formally.

Let $m\cdot e \subset X_\Delta$ be the closed subscheme of the torus $T$ defined by the ideal $(x-1, y-1)^m$. Consider the restriction of the sheaf $\cO_{X_\Delta} (H)$ to the subscheme $m\cdot e$:
\[ \cO_{X_\Delta}(H) \to \cO_{X_\Delta} (H)|_{m\cdot e} = \cO_{m\cdot e},\]
where the last equality follows from the trivialization of the sheaf $\cO_{X_\Delta}(H)$ on the torus. This morphism induces a morphism of global sections
\[ \phi: H^0(X_\Delta, \cO_{X_\Delta}(H)) \to H^0 (X_\Delta, \cO_{m\cdot e}),\]
which can be described more concretely as follows. A global section $s$ on the left hand side is a Laurent polynomial in $x$ and $y$. The map $\phi$ takes it to
\[ \phi: s(x,y) \mapsto s(x+1,y+1) \in \FF_2[x,y]/(x,y)^m.\]

\begin{lemma} \label{lem-cohom}
The kernel and cokernel of $\phi$ are $H^0(X, \cO_X(D))$ and $H^1(X, \cO_X(D))$, respectively.
\end{lemma}

\begin{proof}
There is an exact sequence of sheaves on $X$
\[ 0 \to \cO_X(D) \to \cO_X(\pi^* H) \to \cO_X(\pi^* H)|_{m\cdot E} \to 0,\]
where $m\cdot E \subseteq X$ is the $m$-fold thickening of $E$. This gives an exact sequence of cohomology
\[ 0 \to H^0(X, \cO_X(D)) \to H^0(X, \cO_X(\pi^* H)) \to H^0(X, \cO_X(\pi^* H)|_{m\cdot E}) \to H^1(X, \cO_X(D)) \to 0.\]
The sequence ends on the right because $H$ is a nef $\QQ$-Cartier divisor on a toric variety, hence its higher cohomology vanishes \cite{Fulton}:
\[H^1 (X, \cO_X(\pi^* H)) = H^1(X_\Delta, \cO_{X_\Delta} (H)) = 0.\]
In the exact sequence of global sections, the middle map can be identified with $\phi$.
\end{proof}

\begin{remark} \label{rem-int}
Lemma~\ref{lem-cohom} was stated for the variety $X$ over the field $\FF_2$, but clearly it holds for $X$ over any field. Note also that we can represent the map $\phi$ by an integer matrix that does not depend on the field. This is done by choosing the basis of monomials in both source and target. In Section~\ref{sec-4} below, we will in fact define $\phi$ over any base ring $S$. Working over an arbitrary field, the kernel and cokernel of $\phi$ are defined over the prime field, $\FF_2$ in characteristic $2$ and $\QQ$ in characteristic $0$. Consider the case where $\ker \phi$ is nonzero over a field of characteristic $0$. Then we may choose a nonzero element in the kernel as a polynomial with rational coefficients, and after clearing denominators, we may assume further that the coefficients are integers with $\gcd$ of all coefficients equal to $1$. Such a polynomial defines a nonzero element in $H^0(X, \cO_X(D))$ in characteristic $0$ that reduces to a nonzero element in characteristic $2$. 
\end{remark}

\subsection{The module $M$.}

Recall that we defined the section ring $R$ graded by $(x_L\in \frac{1}{n_L}\ZZ, m\in \ZZ_{\geq 0})$:
\[R = \bigoplus_{x_L, m} H^0(X, \cO_X(x_L D_1 + m D_2)). \]
We define similarly the space of first cohomologies:
\[M = \bigoplus_{x_L, m} H^1(X, \cO_X(x_L D_1 + m D_2)). \]
The graded vector spaces $R$ and $M$ can be put into an exact sequence
\[ 0\to R \to \bigoplus_{x_L, m} H^0(X_\Delta, \cO_{X_\Delta}(H_{x_L,m})) \stackrel{\phi}{\to} \bigoplus_{x_L, m} H^0 (X_\Delta, \cO_{m\cdot e}) \to M \to 0,\]
where $H_{x_L,m}$ is a $T$-Weil divisor on $X_\Delta$ such that 
\[ x_L D_1 + m D_2 = \pi^* H_{x_L,m} -mE.\]

Let now $s$ be a homogeneous element of $R$ in degree $(x_L^0, m^0)$. Then multiplication with $s$ defines a graded morphism of degree $(x_L^0, m^0)$ of the exact sequence above:
\begin{equation} \label{eq-comm}
\begin{tikzcd}[column sep=1.7em]
0  \arrow[r] &  R  \arrow{d}{s(x, y)} \arrow[r] &\bigoplus H^0(X_\Delta, \cO_{X_\Delta}(H_{x_L,m})) \arrow{d}{s(x, y)} \arrow{r}{\phi} &  \bigoplus H^0 (X_\Delta, \cO_{m\cdot e}) \arrow{d}{s(x+1, y+1)}\arrow[r] & M \arrow{d}{s(x+1, y+1)} \arrow[r] & 0 \\
0  \arrow[r] &  R   \arrow[r] &\bigoplus H^0(X_\Delta, \cO_{X_\Delta}(H_{x_L,m}))  \arrow{r}{\phi} &  \bigoplus H^0 (X_\Delta, \cO_{m\cdot e})  \arrow[r] & M  \arrow[r] & 0.
\end{tikzcd}
\end{equation}
These multiplication maps give $M$ the structure of a graded $R$-module. 
Recall that $R=\FF_2[\xi,\zeta,x^{\pm 1}]$, where $\xi$ and $\zeta$ are defined in Lemma~\ref{lem-xi} and Lemma~\ref{lem-zeta}, 
\[ \xi(x+1,y+1) =  g_1(x,y), \quad \zeta(x+1,y+1) = g_2(x,y).\]
The action of $\xi$ and $\zeta$ on $M$ is by multiplication with $g_1$ and $g_2$, respectively. The variable $x$ acts by multiplication with $x+1$. The inverse $x^{-1}$ acts by multiplication with $1-x+x^2-\ldots$. This is well defined because on each graded piece of $M$ multiplication with $x^l$ is  zero for $l$ large enough.

\begin{theorem} \label{thm-free}
The module $M$ is a graded free $R$-module.
\end{theorem}

\subsection{Proof of Theorem~\ref{thm-free}.}

We show that $\xi,\zeta$ form an $M$-regular sequence. This implies that $M$ is a graded free $\FF_2[\xi,\zeta]$-module. The element $x$ acts on $M$ by a graded isomorphism. This action turns $M/(\xi,\zeta)M$ into a free $\FF_2[x^{\pm 1}]$-module, which implies that $M$ is a free $\FF_2[\xi,\zeta, x^{\pm 1}]$-module.

\begin{lemma} \label{lem-reg-xi}
The element $\xi\in R$ is not a zero-divisor on $M$.
\end{lemma}

\begin{proof}
Consider the map of exact sequences (\ref{eq-comm}) defined by $s=\xi$. We restrict this map to one degree so that it goes from degree $(x_L-\alpha,m-3)$ to degree $(x_L, m)$, $m\geq 3$. Write the divisors in these degrees as 
\begin{align*} 
\left(x_L-\alpha,m-3\right) &: D' = \pi^* H' - (m-3) E,\\
(x_L, m) &: D = \pi^* H - m E.
\end{align*}
We add the kernels and cokernels of the vertical maps to the diagram. Then, with some abuse of notation, we have
\[
\begin{tikzcd}[column sep=scriptsize]
 & 0 \arrow[d] & 0 \arrow[d] & 0 \arrow[d] & K \arrow[d] & \\ 
0 \arrow[r] & H^0(\cO(D')) \arrow{d}{\xi} \arrow[r] & H^0(\cO(H')) \arrow{d}{\xi} \arrow{r}{\phi} & H^0 (\cO_{(m-3)\cdot e}) \arrow{d}{g_1}\arrow[r] & H^1(\cO(D')) \arrow{d}{g_1} \arrow[r] & 0 \\
0 \arrow[r] & H^0(\cO(D)) \arrow{d} \arrow[r] & H^0(\cO(H)) \arrow{d} \arrow{r}{\phi} & H^0 (\cO_{m\cdot e}) \arrow{d}\arrow[r] & H^1(\cO(D)) \arrow{d}\arrow[r] & 0 \\
 & Q_1 & Q_2 & Q_3 & Q_4 & 
\end{tikzcd}
\]
It is easy to see that the kernels in the left three columns are all zero. In the second column, the map $\xi$ is multiplication in the Laurent polynomial ring. Since this ring is a domain, the kernel of this map, as well as the kernel in the first column, is zero. In the third column, if $g_1 f \equiv 0 \pmod{(x,y)^m}$, then using the fact that $g_1$ lies in $(x,y)^3 \setmin (x,y)^4$, we get that $f \equiv 0 \pmod{(x,y)^{m-3}}$.

The diagram gives us a sequence of maps of the cokernels
\[ 0\to Q_1 \stackrel{\delta_1}{\to} Q_2 \stackrel{\delta_2}\to Q_3 \to Q_4 \to 0. \] 
This sequence is exact except possibly at $Q_2$,  
\[ \ker \delta_2/ \image\delta_1 \isom K.\]
(This can be seen by cutting the two rows into two short exact sequences and then considering the two long exact sequences of kernels and cokernels.) 
Thus, to show that $K=0$, we need to prove that the sequence is also exact at $Q_2$. We will describe the three spaces $Q_1, Q_2, Q_3$.

The space $Q_1$ is the quotient $\FF_2[\xi,\zeta, x^{\pm 1}]/(\xi)$ in degree $(x_L,m)$. This is either zero or $1$-dimensional with basis $x^l \zeta^{m/7}$ when $x_L=l \in \ZZ$ and $7 \mid m$.

To find $Q_2$, by Lemma~\ref{lem-tor} we have an exact sequence of sheaves
\[ 0 \to \cO_{X_\Delta}(H') \stackrel{\xi}{\to} \cO_{X_\Delta}(H) \to \cO_{X_\Delta}(H)|_{\overline{C}}/\torsion \to 0.\]
Here $\overline{C} = V(\xi) \subseteq X_\Delta$, and $C$ is its strict transform in $X$. We assume that $m-3\geq 0$, hence $H'$ is a nef divisor on $X_\Delta$ and the higher cohomology of $\cO_{X_\Delta}(H')$ vanishes. Then the cohomology sequence of the short exact sequence gives us 
\[ Q_2 = H^0(\overline{C}, \cO_{X_\Delta}(H)|_{\overline{C}}/\torsion).\]
The sheaf $\cO_{X_\Delta}(H)$ is locally free in a neighborhood of $\overline{C}$ if and only if the left and right vertices of the triangle corresponding to $H$ are integral. This is equivalent to $x_L= l \in \ZZ$ and $7 \mid m$. (Notice that this is the same condition as the one for $Q_1$ to be nonzero.) 
In this case $\cO_{X_\Delta}(H)$ restricts to a degree $H\cdot \overline{C} = 3m$ line bundle on $\overline{C}$. 

In general, $n H $ is a Cartier divisor in a neighborhood of $C$ for some integer $n>0$. The sheaf $\cO_{X_\Delta}(n H)$ restricts to a degree $n\cdot 3 m$ line bundle on $\overline{C}$. We will use the inclusion 
\[ \cO_{X_\Delta}( H) ^{\tensor n} \subseteq \cO_{X_\Delta}(n H) \]
as follows. Given an element $s \in Q_2$, $s^{n}$ gives a global section of the line bundle $\cO_{X_\Delta}(n H)|_{\overline{C}}$. If $H$ is not Cartier at a $T$-fixed point on $\overline{C}$, then the section $s^{n}$ vanishes at that point.

The space $Q_3$ is the coordinate ring of the affine scheme $(m\cdot e) \cap \overline{C}$.

Now consider $s\in \ker (Q_2\to Q_3)$. This $s$ is a global section of $\cO_{X_\Delta}(H)|_{\overline{C}}/\torsion$ that vanishes on $m\cdot e \cap \overline{C}$. We claim that the vanishing imposes $3m$ conditions on the sections $s\in Q_2$. Since we have trivialized $\cO_{X_\Delta}(H)$ on the torus, this does not depend on the divisor $H$. We may assume that $H=0$ and work with functions in $\cO_{\overline{C}}$. Functions on $\overline{C}$ can be pulled back to functions on $C$. If a function on $\overline{C}$ vanishes on $(m\cdot e) \cap \overline{C}$, then its pullback vanishes on $(m \cdot E)\cap C$. When working over $\overline{\FF}_2$, the curve $C$ intersects $E$ transversely at $3$ points. Hence the pullback of $s$ must vanish at $3m$ points on $C$ when counted with multiplicities. This imposes $3m$ linear conditions on sections $s\in Q_2$. Note also that the map $\delta_2: Q_2\to Q_3$ is a linear map defined over the field $\FF_2$. Its rank does not change if we extend the field to  $\overline{\FF}_2$.

First consider the case where $H$ is Cartier in a neighborhood of $\overline{C}$. Then the pullback of $\cO_{X_\Delta}(H)$ to $C$ is a line bundle of degree $3m$. The pullback of $s$ gives a section of this bundle that vanishes at $3m$ points. There can be at most one such nonzero section, so that $\ker (Q_2\to Q_3)$ has dimension at most $1$. In this case we know that $Q_1$ has dimension $1$. It follows that $\image(Q_1\to Q_2) = \ker(Q_2\to Q_3)$.

If $H$ is not Cartier in a neighborhood of $\overline{C}$, consider the section $s^{n}$ pulled back to $C$ for some $n>0$ such that $nH$ is Cartier. This pullback is a global section of a degree $n\cdot 3 m$ line bundle that vanishes at $n\cdot 3 m$ points in $C\cap E$. Since $H$ is not Cartier, $s^{n}$ also vanishes at a $T$-fixed point on $C$. This implies that $s=0$. In this case also $Q_1=0$, hence $\image(Q_1\to Q_2) = \ker(Q_2\to Q_3) = 0$.
\end{proof}

\begin{lemma}\label{lem-reg-zeta}
The element $\zeta \in R$ is not a zero-divisor on $M/\xi M$.
\end{lemma}

\begin{proof}
The degree $(x_L,m)$ part of the quotient module $M/\xi M$ was called $Q_4$ in the proof of the previous lemma. In the same lemma we constructed an exact sequence 
\[ 0\to Q_1 \to Q_2\to Q_3\to Q_4 \to 0,\]
and described the other terms $Q_i$ for $i=1,2,3$ explicitly. We will now consider two such exact sequences of $Q_i$ in degrees $(x_L, m-7)$ and $(x_L, m)$. Multiplication with $\zeta$ takes one sequence to the other. On the terms $Q_4$ this gives the action of $\zeta$ on $M/\xi M$. Let us denote the divisors in degrees $(x_L, m-7)$ and $(x_L, m)$
\begin{align*} 
(x_L,m-7) &: B' = \pi^* H' - (m-7) E,\\
(x_L, m) &: B = \pi^* H - m E.
\end{align*}
We will assume that $m\geq 7$, otherwise the source $Q_4$ is zero. The morphism of the two exact sequences with kernels and cokernels added is:
\[
\begin{tikzcd}[column sep=1.55em]
 & 0 \arrow[d] & 0 \arrow[d] & 0 \arrow[d] & L \arrow[d] & \\ 
0 \arrow[r] & (R/\xi)_{B'} \arrow{d}{\zeta} \arrow[r] & H^0(\cO(H')|_{\overline{C}}/\tors) \arrow{d}{\zeta} \arrow{r}{\phi} & H^0 (\cO_{(m-7)\cdot e \cap \overline{C}}) \arrow{d}{\zeta}\arrow[r] & (M/\xi M)_{B'} \arrow{d}{\zeta} \arrow[r] & 0 \\
0 \arrow[r] & (R/\xi)_{B} \arrow{d} \arrow[r] & H^0(\cO(H)|_{\overline{C}}/\tors) \arrow{d} \arrow{r}{\phi} & H^0 (\cO_{m\cdot e \cap \overline{C}}) \arrow{d}\arrow[r] & (M/\xi M)_{B} \arrow{d} \arrow[r] & 0 \\
 & C_1 & C_2 & C_3 & C_4 & 
\end{tikzcd}
\]
Here we have used subscripts $B, B'$ to denote the degree. Again, it is not hard to check that the kernels of the left three vertical maps are zero. In the second column this follows because we are multiplying sections of torsion-free sheaves on an irreducible curve. 

To prove that the third vertical map $\zeta$ is injective, we will change the ground field from $\FF_2$ to $\overline{\FF}_2$, which allows a more geometric argument. Since $\zeta$ is a linear map between $\FF_2$-vector spaces, its injectivity does not change if we extend the field. We work over $\overline{\FF}_2$ and consider regular functions on $\overline{C}$ pulled back to regular functions on $C$. This defines an inclusion
\[ H^0 (\overline{C}, \cO_{m\cdot e \cap \overline{C}}) \subseteq H^0 (C, \cO_{m\cdot E \cap C}).\]
The pullback of $\zeta$ to the blowup defines the divisor $D+7E$, and since $D$ does not intersect $C$, multiplication by the pullback of $\zeta$ is the multiplication by a local equation of $7E$. This multiplication is injective. Locally at each of the $3$ points in $C \cap E$, the map $\zeta$ is isomorphic to
\[ t^7: \overline{\FF}_2[t]/(t^{m-7}) \to \overline{\FF}_2[t]/(t^{m}),\]
where $t$ is the local equation of $E$, defining a local parameter for $C$.

We also note that the cokernel $C_1 = R/(\xi,\zeta) = \FF_2[x^{\pm 1}]$ is zero in degree $B$ because $m>0$. This gives us an exact sequence 
\[ 0\to L \to C_2 \stackrel{\delta}\to C_3 \to C_4 \to 0, \]
and we need to prove that $\delta$ is injective. We do this by describing $C_2$ and $C_3$ geometrically. 

Let $\overline{D}\subseteq X_\Delta$ be the curve defined by $\zeta$, so that $D$ is its strict transform in $X$. The space $C_3$ is the coordinate ring of the affine scheme $(m\cdot e) \cap \overline{C} \cap \overline{D}$. This scheme is equal to $\overline{C} \cap \overline{D}$. Indeed, in the blowup $X$, the scheme defined by the vanishing of $\xi$ and $\zeta$ is 
\[ (C\cup 3E)\cap (D\cup 7E) = (C\cap 7E) \cup (D\cap 3E) \cup 3E \subseteq 7E \subseteq mE.\]

To find the space $C_2$, we start with the exact sequence 
\[ 0 \to \cO_{X_\Delta}(H') \stackrel{\zeta}{\to} \cO_{X_\Delta}(H) \to \cO_{X_\Delta}(H)|_{\overline{D}}/\torsion \to 0.\]
Restriction of these sheaves to $\overline{C}$ gives an exact sequence
\[
  0 \to \cO_{X_\Delta}(H')|_{\overline{C}} \stackrel{\zeta}{\to} \cO_{X_\Delta}(H)|_{\overline{C}} \to \cO_{X_\Delta}(H)|_{\overline{D}\cap {\overline{C}}} \to 0.
\]
There is no $\operatorname{Tor}_1$ term on the left because $\overline{D}\cap {\overline{C}}$ is a complete intersection in $X_\Delta$. There is also no need to mention the torsion on the right because $\overline{D}\cap {\overline{C}}$ is supported at the point $e$ while the torsion was supported at a $T$-fixed point. The left two sheaves may have torsion at the two $T$-fixed points that lie on $\overline{C}$. We can mod out by this torsion and consider the long exact sequence in cohomology:
\[
  0 \to H^0(\overline{C},\cO_{X_\Delta}(H')|_{\overline{C}}/\tors) \stackrel{\zeta}{\to} H^0(\overline{C},\cO_{X_\Delta}(H)|_{\overline{C}}/\tors) \to H^0(\overline{C},\cO_{X_\Delta}(H)|_{\overline{D}\cap {\overline{C}}}) \to \cdots
\]
The cokernel of $\zeta$ in this sequence gives the space $C_2$. Note that 
\[ \coker \zeta \subseteq H^0(\overline{C},\cO_{X_\Delta}(H)|_{\overline{D}\cap {\overline{C}}}) = H^0(\overline{C},\cO_{\overline{D}\cap {\overline{C}}}) = C_3.\]
This inclusion is the map $\delta$ in the sequence above.
\end{proof}

\section{Proof of Theorem~\ref{thm-main}.} \label{sec-4}

We fix one variety $X$. Suppose by contradiction that $X$ over a field of characteristic $0$ contains a negative curve. As explained in Remark~\ref{rem-int}, we may assume that this curve is defined by a polynomial $\tilde{s}$ with integer coefficients such that the reduction $s$ of $\tilde{s}$ mod $2$ gives a possibly reducible negative curve in characteristic $2$. We know that all such $s$ must lie in the ring $R$, hence $\tilde{s}$ is a lift of the element $s\in R$ to characteristic $0$. To prove Theorem~\ref{thm-main}, we show that no element of $R$ in degree $m>0$ can be lifted to characteristic zero.

For any ring S, define $R_S$ and $M_S$ by the exact sequence of graded maps of $S$-modules
\[ 0\to R_S \to \bigoplus_{x_L, m} H^0(X_{\Delta,S}, \cO_{X_{\Delta,S}}(H_{x_L,m})) \stackrel{\phi}{\to} \bigoplus_{x_L, m} H^0 (X_{\Delta,S}, \cO_{m\cdot e, S}) \to M_S \to 0.\]
Here $H^0(X_{\Delta,S}, \cO_{X_{\Delta,S}}(H_{x_L,m}))$ consists of Laurent polynomials with coefficients in $S$ and supported in the triangle corresponding to the degree $(x_L, m)$. In the third term, 
\[ H^0 (X_{\Delta,S}, \cO_{m\cdot e, S}) = S[x,y]/(x,y)^m,\]
and $\phi$ is the change of coordinates map
\[ \phi: s(x,y) \mapsto s(x+1,y+1) \pmod{(x,y)^m}. \]
We are interested in the cases where $S=\ZZ$ or $S=\ZZ/2^p\ZZ$. The case $p=1$ gives the previously defined $R$ and $M$. 

The definition of $R_S$ is functorial in $S$. We say that an element $r\in R = R_{\ZZ/2 \ZZ}$ lifts to characteristic zero if $r$ lies in the image of $R_\ZZ \to R$. More generally, given $r\in R_{\ZZ/2^p \ZZ}$, we say that $r$ lifts mod $2^q$ for $q>p$ if it lies in the image of $ R_{\ZZ/2^q \ZZ} \to R_{\ZZ/2^p \ZZ}$. Clearly, if an element $r\in R$ lifts to characteristic zero then it lifts mod $2^q$ for any $q>1$. 

For an integer $a$, define its $2$-adic order, $\ord(a)$, as the largest $p\geq 0$ such that $2^p$ divides $a$. The element $0$ has order $\infty$. For a polynomial with integer coefficients, we define similarly its order as the largest $p$ such that $2^p$ divides all its coefficients. If the polynomial has order $p$, we call its initial terms the sum of those terms that have coefficients with order exactly $p$.  A polynomial has infinite order if and only if it is the zero polynomial. In that case its initial terms are also zero.

\begin{lemma} \label{lem-lift}
Let $\tilde{s} \in H^0(X_{\Delta,\ZZ}, \cO_{X_{\Delta,\ZZ}}(H_{x_L,m}))$ be such that $\tilde{s}$ lies in the kernel of $\phi \pmod {2^p}$:
\[ \phi(\tilde{s}) = 2^p f,\]
for some $p> 0$, $f\in \ZZ[x,y]/(x,y)^m$, and hence $\tilde{s}$ defines an element $s \in R_{\ZZ/2^p\ZZ}$. Then $s$ can be lifted to mod $2^{p+1}$ if and only if $f$ maps to zero in $M$.
\end{lemma}

\begin{proof} 
To lift $s$ means to find $\tilde{s}'\in H^0(X_{\Delta,\ZZ}, \cO_{X_{\Delta,\ZZ}}(H_{x_L,m}))$ such that 
\[ \phi (\tilde{s}+ 2^{p} \tilde{s}' ) = 2^p (f + \phi(\tilde{s}')) \equiv 0 \pmod {2^{p+1}}.\]
This congruence holds if and only if 
\[ f \equiv -\phi(\tilde{s}') \pmod 2,\]
which means $f=0$ in $M$.
\end{proof}

We start by considering liftings of the generators $\xi$ and $\zeta$ of $R$. For the rest of this section let us fix representatives $\tilde{\xi}, \tilde{\zeta} \in \ZZ[x,y]$, and write these as 
\begin{align*}
\tilde{\xi}(x+1,y+1) = 2 f_1 +g_1,\\
\tilde{\zeta}(x+1, y+1) = 2 f_2 +g_2,
\end{align*}
where $2f_1 = \phi(\tilde{\xi})$, $2f_2 = \phi(\tilde{\zeta})$, $g_1\in (x,y)^3$, $g_2\in (x,y)^7$. Here we view $f_1, f_2$ as elements of the same degree as $\xi, \zeta$, either in $\ZZ[x,y]/(x,y)^m$, where $m=3,7$, or in the quotient $M$.  Multiplication with $g_1, g_2$ gives the action of $\xi,\zeta$.

From the previous lemma we get that $\xi$ (resp. $\zeta$) can be lifted mod $4$ if and only if the image of $f_1$ (resp. $f_2$) is zero in $M$. 

Let us also consider liftings of $\xi^2$ (the case of $\zeta^2$ is similar). We use the same integral representation $\tilde{\xi}$. Then
\[ \tilde{\xi}^2(x+1,y+1) = (2 f_1 +g_1)^2 = 4(f_1^2 + f_1 g_1) + g_1^2. \]
Here $g_1^2 \in (x,y)^6$, hence the first term tells us that $\xi^2$ automatically lifts mod $4$. The image of $f_1^2 + f_1 g_1$ in $M$ determines if $\tilde{\xi}^2 \pmod 4$ can be lifted mod $8$:
 \[ \phi( \tilde{\xi}^2 + 4 \tilde{s}) \equiv 0 \pmod 8.\]
One can ask if, more generally, we can lift $\xi^2$ mod $8$ by adding a multiple of $2$ to $\tilde{\xi}^2$:
\[ \phi( \tilde{\xi}^2 + 2 \tilde{s}' + 4 \tilde{s}) \equiv 0 \pmod 8.\]
However, notice that $\phi(\tilde{s}') \equiv 0 \pmod 2$, and since $\xi^2$ is the only nonzero element of $R$ in its degree, we must have $\tilde{s}'\equiv c \tilde{\xi}^2 \pmod 2$ for some $c\in \FF_2$. We saw that $\phi(\tilde{\xi}^2)\equiv 0 \pmod 4$, hence $\phi(2 \tilde{s}') \equiv 0 \pmod 8$. Therefore, adding the term $2 \tilde{s}'$ does not change the liftability of $\xi^2$ mod $8$. In summary, $\xi^2$ can be lifted mod $8$ if and only if $f_1^2+g_1 f_1$ in zero in $M$.

\begin{lemma} \label{lem-xi-zeta}
The elements $f_1, f_2, f_1^2 + f_1 g_1,f_2^2 + f_2 g_2$ are nonzero in $M$. In other words, $\xi,\zeta$ can not be lifted mod $4$, and $\xi^2, \zeta^2$ can not be lifted mod $8$.
\end{lemma}

We will prove Lemma~\ref{lem-xi-zeta} in the next section. The proof of this lemma is the only part in the proof of Theorem~\ref{thm-main} that depends on the triangle $\Delta$. All other proofs work the same way for all triangles.

\begin{theorem}
Let $s\in R$ be a nonzero homogeneous element of degree $(x_L, m>0)$ that lifts to characteristic zero. Then, there is a nontrivial linear relation:
\[ r_1 x_1 + r_2 x_2 = 0, \quad r_1, r_2 \in R,\]
between elements $x_1, x_2 \in M$, where 
\[ x_1 \in \{ f_1, f_1^2+g_1f_1\}, \quad x_2 \in \{ f_2, f_2^2+g_2f_2\}.\]
\end{theorem}

\begin{proof} 
Let 
\[ s = \sum_{l,i,j} {a}_{l,i,j} x^l {\xi}^i {\zeta}^j, \quad {a}_{l,i,j} \in \FF_2,\]
and let  $\tilde{s}$ be a lifting of $s$ to characteristic zero; that means, an integral polynomial in $H^0(X_{\Delta,\ZZ}, \cO_{X_{\Delta,\ZZ}}(H_{x_L,m}))$ that lies in the kernel of $\phi$ and reduces to $s$ mod $2$. We lift the coefficients ${a}_{l,i,j}$ of $s$ to $\tilde{a}_{l,i,j} \in \ZZ$ and write 
\[ \tilde{s}_1 = \sum_{l,i,j} \tilde{a}_{l,i,j} x^l \tilde{\xi}^i \tilde{\zeta}^j, \quad \tilde{a}_{l,i,j} \in \ZZ.\]
The liftings $\tilde{a}_{l,i,j} $ are chosen so that $\tilde{s}_1$, written as a Laurent polynomial in $x,y$, is as close to $\tilde{s}$ as possible, namely the difference 
\[ \tilde{s}_2 = \tilde{s} - \tilde{s}_1\]
has maximal order. We first check that such an $\tilde{s}_2$ with maximal order exists: if there is a sequence of $\tilde{s}_2$ with increasing orders then we may also take $\tilde{s}_2=0$ with infinite order.  Let $S_1$ be the free $\ZZ$-module generated by the finite set of monomials $x^l \tilde{\xi}^i \tilde{\zeta}^j$ in degree $(x_L, m>0)$, and let
\[ \iota: S_1 \to \ZZ[x^{\pm 1}, y^{\pm 1}]\]
be the $\ZZ$-linear map that writes a polynomial in $x,\tilde{\xi},\tilde{\zeta}$ as a Laurent polynomial in $x,y$. Then $\iota$ is injective. We showed this in the proof of Theorem~\ref{thm-R} over $\FF_2$, and this implies it over $\ZZ$. Now suppose we have a sequence of $\tilde{s}_1\in S_1$ such that the sequence $\iota(\tilde{s}_1)$ approaches $\tilde{s}$ in the $2$-adic topology (the differences $\tilde{s}_2 = \tilde{s} - \iota(\tilde{s}_1)$ have orders approaching infinity). Since the subspace $\iota(S_1) \subseteq \ZZ[x^{\pm 1}, y^{\pm 1}] $ is closed in the $2$-adic topology, the sequence $\tilde{s}_1$ also approaches a limit $\tilde{s}_{1,\infty} \in S_1$. Now if we take $\tilde{s}_1=  \tilde{s}_{1,\infty}$, then $\iota(\tilde{s}_1) = \tilde{s}$ and $\tilde{s}_2 = 0$. 

The orders satisfy:
\begin{enumerate}
\item $\ord(\tilde{s}_2) >0$ because $\tilde{s}_1 \equiv \tilde{s} \pmod{2}$. 
\item $\ord(\phi(\tilde{s}_1)) = \ord(\phi(\tilde{s}_2))$ because $\phi(\tilde{s})$ has infinite order, so the initial terms of $\phi(\tilde{s}_1)$ and $\phi(\tilde{s}_2)$ must cancel.
\item $ord (\phi (\tilde{s}_2)) \geq ord (\tilde{s}_2)$. Change of coordinates does not change the order, but when applying $\phi$ we also remove the higher degree terms.
\item $ord (\tilde{s}_2) \geq \ord(\phi(\tilde{s}_2))$. This is clear if $ord(\tilde{s}_2) = \infty$. Otherwise, let $p=ord (\tilde{s}_2)$, $\tilde{s}_2 = 2^p \hat{s}_2$ for some $\hat{s}_2$. If $p< \ord(\phi(\tilde{s}_2))$ then $\phi(\hat{s}_2) \equiv 0 \pmod 2$. Hence $\hat{s}_2 \pmod 2$ is a nonzero element of $R$. Since $\hat{s}_2 \pmod 2$ is a polynomial in $x^{\pm 1}, \xi, \zeta$, we can move such terms from $\tilde{s}_2$ to $\tilde{s}_1$ and increase the order of $\tilde{s}_2$.
\end{enumerate}
These conditions imply 
\[ \ord(\phi(\tilde{s}_1)) = \ord(\phi(\tilde{s}_2)) = \ord(\tilde{s}_2)>0.\]
We now consider the initial terms in the expression $\phi(\tilde{s}) = \phi(\tilde{s}_1) + \phi(\tilde{s}_2)$. Since the left hand side is zero, the initial terms of $\phi(\tilde{s}_1)$ and $\phi(\tilde{s}_2)$ must sum to zero. We claim that $\phi(\tilde{s}_1)$ has finite order $p$ and 
\[ \phi(\tilde{s}_1) \equiv 2^p (r_1 x_1 + r_2 x_2) \pmod{2^{p+1}},\]
where the linear combination is as in the statement of the theorem. (For simplicity, here and below we are abusing notation slightly. The sum in the parentheses is an element of $\ZZ[x,y]/(x,y)^m$ that mod $2$ reduces to the claimed expression.) 
Assuming this, $\tilde{s}_2$ also has the same finite order $p$, and we can write $\tilde{s}_2 = 2^p \hat{s}_2$. The theorem now follows from the congruence
\[ 0= \frac{1}{2^p}(\phi(\tilde{s}_1) + \phi(\tilde{s}_2)) \equiv   r_1 x_1 + r_2 x_2 + \phi(\hat{s}_2) \pmod 2.\]

We prove the claim about the initial terms of $\phi(\tilde{s}_1)$ by the following steps:
\begin{enumerate}
\item For every monomial $\tilde{r}= x^l \tilde{\xi}^i \tilde{\zeta}^j$ that can appear in $\tilde{s}_1$, we show that $\phi(\tilde{r})$ has a finite order $q$ that depends on $l,i,j$. Moreover,
 \[\phi(\tilde{r}) \equiv 2^q (r_1 x_1 + r_2 x_2) \pmod {2^{q+1}},\]
 where the linear combination mod $2$ is as in the statement of the theorem. 
 \item For different monomials $\tilde{r}$, the expression $r_1 x_1 + r_2 x_2$ has different $r_1, r_2$ but the same $x_1, x_2$. A linear combination of such expressions is then an expression of the same type.
\item The coefficients $(r_1, r_2) \in R^2$ that appear for different monomials $\tilde{r}$ are linearly independent over $\FF_2$. This implies that a nontrivial linear combination of the expressions $r_1 x_1 + r_2 x_2$ for different monomials $\tilde{r}$ is a nontrivial linear combination of $x_1$ and $x_2$. The initial terms of  $\phi(\tilde{s}_1)$ are then $2^p$ times one such nontrivial linear combination, where $p>0$ is finite. 
\end{enumerate}

Consider one monomial $x^l \tilde{\xi}^i \tilde{\zeta}^j$. We will find the initial terms of $\phi(x^l \tilde{\xi}^i \tilde{\zeta}^j) = \phi(x)^l \phi(\tilde{\xi})^i \phi(\tilde{\zeta})^j$. 
 The order of $\phi(x)^l = (x+1)^l$ is $0$. Its initial terms mod $2$ are $(x+1)^l \pmod 2$. Next consider $\phi(\tilde{\xi})^i$:
\[ \phi(\tilde{\xi})^i = (g_1+2f_1)^i = g_1^i + 2 i g_1^{i-1} f_1 + 4 { i\choose 2} g_1^{i-2} f_1^2 + \ldots.\]
When $i$ is odd, we get 
\begin{align*} \phi(\tilde{\xi})^i \equiv g_1^i + 2 i g_1^{i-1} f_1 \pmod 4 \\
\equiv g_1^i + 2 g_1^{i-1} f_1 \pmod 4, 
\end{align*}
and when $i$ is even, we get
\begin{align*} \phi(\tilde{\xi})^i &\equiv g_1^i + 2 i g_1^{i-1} f_1 + 4 { i\choose 2} g_1^{i-2} f_1^2 \pmod {2^{\ord(i)+2}} \\
 &\equiv g_1^i + 2^{\ord(i)+1} g_1^{i-2} ( g_1 f_1 + f_1^2) \pmod {2^{\ord(i)+2}}.
 \end{align*}
 
There are similarly two cases for $\phi(\tilde{\zeta})^j$ depending on whether $j$ is odd or even. When computing the initial terms of $\phi(x^l \tilde{\xi}^i \tilde{\zeta}^j)$, we get four cases depending on the parity of $i$ and $j$. In the first case, where $i$ and $j$ are both odd, 
 \[ \phi(x^l \tilde{\xi}^i \tilde{\zeta}^j) \equiv (x+1)^l (g_1^i + 2 g_1^{i-1} f_1) ( g_2^j + 2 g_2^{j-1} f_2) \pmod 4.\]
 Since $g_1^i g_2^j \equiv 0 \pmod{(x,y)^{m}}$, we get the initial terms
 \[ \phi(x^l \tilde{\xi}^i \tilde{\zeta}^j) \equiv 2 (x+1)^l [g_1^{i-1} g_2^j f_1 + g_1^i g_2^{j-1} f_2 ] \pmod 4.\]
 Here we view $(x+1)^l [g_1^{i-1} g_2^j f_1 + g_1^i g_2^{j-1} f_2 ]$ as a linear combination of elements $x_1=f_1$ and $x_2=f_2$ with coefficients $r_1= (x+1)^l g_1^{i-1} g_2^j$ and $ r_2= (x+1)^l g_1^{i}g_2^{j-1}$. The polynomials $r_1, r_2$ correspond to elements $r_1 = x \xi^{i-1} \zeta^j$ and $ r_2= x^l \xi^{i}\zeta^{j-1}$ in $R$.
 
In the second case, when $i$ is even and $j$ is odd, then 
\begin{align*}
\phi(x^l \tilde{\xi}^i \tilde{\zeta}^j) &\equiv (x+1)^l (g_1^i + 2^{\ord(i)+1} g_1^{i-2} ( g_1 f_1 + f_1^2))( g_2^j + 2 g_2^{j-1} f_2) \pmod 4 \\
&\equiv 2 (x+1)^l g_1^i g_2^{j-1} f_2 \pmod 4.
\end{align*}
Here we have a linear combination with only one term $r_2 x_2$, where $r_2= x^l \xi^i \zeta^{j-1}$ and $x_2=f_2$. The case $i$ odd, $j$ even, gives similarly a one term linear combination.

Finally consider the case where both $i$ and $j$ are even. If the orders of $i$ and $j$ are not equal, then we get a linear combination with one term. Let us do the more complex case when $p=\ord(i)=\ord(j)$. Then 
\begin{align*}
\phi(x^l \tilde{\xi}^i \tilde{\zeta}^j) &\equiv (x+1)^l (g_1^i + 2^{p+1} g_1^{i-2} ( g_1 f_1 + f_1^2))( g_2^j + 2^{p+1} g_2^{j-2} ( g_2 f_2 + f_2^2) ) \pmod {2^{p+2}} \\
&\equiv 2^{p+1} (x+1)^l [ g_1^{i-2} g_2^j (g_1 f_1 + f_1^2) + g_1^{i} g_2^{j-2} (g_2 f_2 + f_2^2) \pmod {2^{p+2}}. 
\end{align*}
Here we have a linear combination of $x_1=g_1 f_1 + f_1^2$ and $x_2=g_2 f_2 + f_2^2$ with coefficients $r_1 = x^l \xi^{i-2} \zeta^j$ and $r_2= x^l \xi^{i} \zeta^{j-2}$.  (When the orders of $i$ and $j$ are not equal, we get one half of this linear combination, $r_1 x_1$ or $r_2 x_2$.)

Notice that all monomials $x^l \tilde{\xi}^i \tilde{\zeta}^j$ in the same degree $(x_L,m)$ must have the same parity of $i$ and the same parity of $j$. This can be seen by finding all Laurent monomials of degree $(0,0)$. These monomials are all powers of $x^3 \xi^{14} \zeta^{-6}$. The fact that $14$ and $6$ are even then proves the claim. This implies that the expressions $r_1 x_1+r_2 x_2$ that we get for each monomial $x^l \tilde{\xi}^i \tilde{\zeta}^j$ involve the same $x_1$ and $x_2$. Hence a linear combination of these expressions is an expression of the same type.

Finally, let us prove that the coefficients $(r_1, r_2) \in R^2$ that appear in the expression $r_1 x_1+r_2 x_2$ for different monomials $x^l \tilde{\xi}^i \tilde{\zeta}^j$ are linearly independent over $\FF_2$. This is easy to see from the formulas above. For example, suppose the degree $(x_L, m)$ is such that both $i$ and $j$ are even. Then the coefficients $(r_1, r_2)$ for a monomial $x^l \tilde{\xi}^i \tilde{\zeta}^j$ are either $(x^l \xi^{i-2} \zeta^j, x^l \xi^{i} \zeta^{j-2})$ or $(x^l \xi^{i-2} \zeta^j, 0)$ or $(0, x^l \xi^{i} \zeta^{j-2})$. These three cases correspond to $ord(i)=ord(j)$, $ord(i)<ord(j)$, and $ord(i)>ord(j)$. Since the monomials $x^a \xi^b \zeta^c$ are linearly independent in $R$, these vectors in $R^2$ as $l,i,j$ vary are clearly linearly independent.
 \end{proof}

Recall that $R$ and $M$ are graded by $(x_L,m)$. In the next theorem we only need the second component $m$ of the degree. Let us simply call this the degree.

\begin{theorem}
Let $x_1, x_2 \in M$ be nonzero homogeneous elements, $x_1$ in degree $m=3$ or $m=6$, and $x_2$ in degree $m=7$ or $m=14$.
Then there is no nontrivial relation 
\[ r_1 x_1 + r_2 x_2 = 0\]
in $M$ with coefficients $r_1, r_2 \in R$.
\end{theorem}

\begin{proof}
We will use the fact that $M$ is a free $R$-module. In particular, it is a flat $R$-module. The flatness criterion applied to the relation above states that there exist elements $y_j\in M$ and $a_{i,j}\in R$ such that 
\[ \sum_i r_i a_{i,j} = 0, \quad x_i = \sum_j a_{i,j} y_j.\]
Since $x_i$ are homogeneous, we may assume that the relation above is also homogeneous. Then we may choose $y_j$ and $a_{i,j}$ homogeneous such that if $a_{i,j}\neq 0$ then $\deg(x_i) = \deg(a_{i,j})+\deg(y_j)$. 

Consider the degrees of nonzero elements in $M$ and $R$. First we claim that $M$ is zero in degree $m=0$. Indeed, $M$ is a quotient of $\bigoplus_{x_L, m} H^0 (X_\Delta, \cO_{m\cdot e})$, and this space is zero in degree $m=0$. The ring $R$ has nonzero elements in degrees $m=0,3,6,7, 9, 10, 12, 13, 14, \ldots$.

For each $x_i$ consider the possible degrees of $y_j$ that are involved in the equality $x_i = \sum a_{i,j} y_j$; that means, $y_j$ such that $a_{i,j}\neq 0$. Considering the possible degrees of $a_{i,j}\in R$, we get the possible degrees $m$ of $y_j$ as:
\begin{align*}
 x_1 &: m=3, 6,\\
 x_2 &: m=1, 2, 4, 5, 7, 8, 11, 14.
\end{align*}
Since no $y_j$ can appear in both $x_1$ and $x_2$, it follows that $r_i a_{i,j}=0$ for all $i,j$. This implies that both $r_1=0$ and $r_2=0$.
\end{proof}

The previous two theorems complete the proof of Theorem~\ref{thm-main}.

\section{computations} \label{sec-5}

We now prove Lemma~\ref{lem-xi-zeta}. The proof reduces to linear algebra computations. Since the matrices are large (up to size $105\times 105$), these are best done on a computer. We have used the Sage system \cite{SageMath}.

To explain the idea, let us prove the lemma first for a specific $X$ defined by $\alpha=-3/14$. We will then do the same for all $\alpha$.

Let us start by proving that $f_1$ is not zero in $M$. By Lemma~\ref{lem-lift}, this is equivalent to $\xi$ not having a lift mod $4$. Since $\xi$ is the only nonzero element in $R=R_{\ZZ/2\ZZ}$ in its degree, it is also equivalent to saying that the only elements of $R_{\ZZ/4\ZZ}$ in the same degree as $\xi$ are divisible by $2$. 

Let $L$ be the $6\times 6$ integer matrix that represents $\phi$ in the degree of $\xi$ over $\ZZ$. We need to show that every integer vector $v$ such that $Lv \equiv 0 \pmod 4$ is divisible by $2$. Recall that the matrix $L$ can be put in the Smith normal form (by viewing $L$ as a linear transformation $\ZZ^6\to \ZZ^6$ and choosing bases for the source and the target) so that the matrix becomes diagonal. The diagonal entries are the elementary divisors of $L$. These can be computed in Sage with the command $\texttt{L.elementary\_divisors()}$. We know that the matrix $L$ has a $1$-dimensional kernel mod $2$, hence there is one elementary divisor divisible by $2$. We need to check that this divisor is not divisible by $4$. This can be verified by computing the elementary divisors:
\[ \texttt{[1, 1, 1, 1, 1, 2]}.\]

The matrix $L$ is constructed as follows. It has columns indexed by lattice points $(a,b)$ in the triangle $\Delta$ and rows indexed by lattice points $(i,j)$ in the triangle with vertices $(0,0), (m-1,0),(0,m-1)$. The entry in row $(i,j)$ and column $(a,b)$ is 
\[ {a \choose i} {b \choose j}.\]
In Sage it is convenient to give a polygon by its vertices and then list all its lattice points with a single command, for example:
\small
\[ \texttt{Polyhedron(vertices=[(-3/14,0),(15/14,0),(3,7)]).integral\_points()}\]
\normalsize

Next, $f_1^2+g_1 f_1$ is nonzero in $M$ if and only if $\xi^2$ cannot be lifted mod $8$. As in the case of $\xi$, this is equivalent to saying that all elements of $R_{\ZZ/8\ZZ}$ in the same degree as $\xi^2$ are divisible by $2$. Let $L$ be the matrix representing $\phi$ in the degree of $\xi^2$. The elementary divisors of $L$ computed in Sage are:
\small
\[ \texttt{[1, 1, 1, 1, 1, 1, 1, 1, 1, 1, 1, 1, 1, 1, 1, 1, 1, 1, 1, 4, 0]}.\]
\normalsize
The matrix $L$ in this case has size $21\times 20$. If the matrix has more rows than columns then the list of diagonal entries in its Smith normal form is padded with zeros. These zeros correspond to the free part of $\coker(L)$. Since we are interested in the kernel of $L$, we may ignore these trailing zeros. 
There is exactly one (nonzero) elementary divisor divisible by $4$, confirming that $\phi(\tilde{\xi}^2)\equiv 0 \pmod 4$. Since this elementary divisor is not divisible by $8$, it follows that every element in the kernel of $L$ mod $8$ is divisible by $2$, hence $\xi^2$ cannot be lifted mod $8$.

Similarly, the proof that $f_2\neq 0$ and $f_2^2+g_2f_2 \neq 0$ in $M$ is reduced to computing the elementary divisors. In the case of $f_2$ the matrix is $28\times 28$ and has elementary divisors:
\[
 \texttt{[1,..., 1, 1, 14]},
\]
\noindent and in the case of $f_2^2+g_2f_2$:
\[ \texttt{[1,..., 1, 1, 28, 0, 0, 0]}. \]
In the last case, the matrix $L$ has size $105\times 102$ and just as before we may ignore the trailing zeros. The fact that $14$ is not divisible by $4$ and $28$ is not divisible by $8$ proves that $f_2$ and $f_2^2+g_2f_2$ are nonzero in $M$. 

\subsection{General triangles $\Delta$.}

We will now prove Lemma~\ref{lem-xi-zeta} for all $X = \Bl_e X_\Delta$ in Theorem~\ref{thm-main}. 

It suffices to prove that $\xi^2$ and $\zeta^2$ do not lift mod 8 no matter which triangle $\Delta$ we consider. The case of $\xi$ and $\zeta$ follows from this. For example, if $\xi$ lifts mod $4$ then by the same argument as in Section~\ref{sec-4}, $\xi^2$ automatically lifts mod $8$.

Let us start with $\xi^2$. We need to show that the triangle of $\xi^2$ does not support an integral polynomial  $s$ such that $\phi(s) \equiv 0 \pmod{8}$ and $s\not\equiv 0 \pmod 2$. It suffices to prove this for a larger polygon that contains the triangle of $\xi^2$. Indeed, if we can find such $s$ supported in the triangle, then the same $s$ is also supported in the larger polygon. We can thus choose a polygon that contains some (or even all) of the triangles that we are considering. Then proving that the polygon does not support such $s$ proves it for all these triangles.

We are interested in triangles with $\alpha\in [-4/15, -16/105]$. We divide this set of $\alpha$ into two subintervals:
\[ [-4/15, -3/14], [-3/14, -16/105].\]
The triangles of $\xi$ for the first set of $\alpha$ all lie in a larger triangle with vertices
\[ (-4/15,0), (15/14,0), (3,7),\]
and for the second set of $\alpha$ they lie in a larger triangle with vertices
\[ (-3/14,0), (17/15,0), (3,7).\]
We will now consider the map $\phi$ from polynomials supported in twice these two larger triangles to $\ZZ[x,y]/(x,y)^6$, and compute its elementary divisors. In both cases the matrix $L$ has size $21\times 21$ and the elementary divisors are the same:
\[
 \texttt{[1, ...,1,4]}
 \]
Since $8$ does not divide any elementary divisor, this shows that $\xi^2$ can not be lifted mod $8$ for any triangle $\Delta$.

Next we do the same computation for $\zeta^2$. The triangles of $\zeta$ have vertices 
\[ (0,0), (0,3), (x_T/45,  49/3),\]
where $x_T$ ranges in the interval $[331, 343]$. We divide this interval into $4$ subintervals with vertices at
\[ x_0 = 331, x_1 = 335, x_2=339, x_3=341.5, x_4=343.\]
Then the triangles corresponding to the $i$-th subinterval, $i=1,\ldots, 4$, lie in a trapezoid with vertices
\[ (0,0), (3,0), (x_{i-1}/45, 49/3),(x_i/45, 49/3).\]
We consider the map $\phi$ from polynomials supported in twice the trapezoid to $\ZZ[x,y]/(x,y)^{14}$. We list the size of the matrix $L$ and the elementary divisors in each of the four cases:
\begin{align*} 
105\times 104,  & \texttt{[1,..., 1,1,28,0]},\\
105\times 104,  & \texttt{[1,..., 1,1,15652,0]},\\
105\times 105,  & \texttt{[1,..., 1,1,42028]},\\
105\times 103,  & \texttt{[1,..., 1,1,4,0,0]}.
\end{align*}
The size of the matrices tells us that we may ignore the trailing zeros. Since the nonzero elementary divisors are not divisible by $8$, this shows that $\zeta^2$ cannot be lifted mod $8$ for any triangle $\Delta$. \qed

%\bibliographystyle{plain}
%\bibliography{cox}

\end{document}